\documentclass[a4paper,11pt]{amsart}
\usepackage{amsmath,amsfonts,amsthm,amssymb,enumerate}
\usepackage{graphicx}

\numberwithin{equation}{section}
\newtheorem{theorem}{Theorem}[section]
\newtheorem{corollary}[theorem]{Corollary}
\newtheorem{proposition}[theorem]{Proposition}

{
\theoremstyle{definition}
\newtheorem{remark}[theorem]{Remark}

\newtheorem{definition}[theorem]{Definition}
\newtheorem{example}[theorem]{Example}

}

\newenvironment{df}{\begin{definition}\rm}{\end{definition}}
\newenvironment{rem}{\begin{remark}\rm}{\end{remark}}

\newcommand{\R}{\mathbb R}

\newcommand{\Z}{\mathbb Z}

 \newcommand{\mS}{\mathcal{S}}
\newcommand{\be}{\mathbf{e}}
\newcommand{\tmA}{\widetilde{\mathcal{A}}}
\DeclareMathOperator{\Tlk}{Tlk}
\DeclareMathOperator{\Dlk}{Dlk}
\DeclareMathOperator{\lk}{lk}
\DeclareMathOperator{\Aut}{Aut}

\title{On surface links whose link groups are abelian}
\author{Tetsuya Ito}
\author{Inasa Nakamura}
\address{
Tetsuya Ito\newline
Research Institute for Mathematical Sciences, Kyoto University\newline Kyoto, 606-8502, Japan}
\email{tetitoh@kurims.kyoto-u.ac.jp}
\urladdr{http://www.kurims.kyoto-u.ac.jp/\~{}tetitoh/}
\address{Inasa Nakamura \newline
Institute for Biology and Mathematics of Dynamical Cell Processes (iBMath), Interdisciplinary Center for Mathematical Sciences, Graduate School of Mathematical Sciences, The University of Tokyo\newline
3-8-1 Komaba, Tokyo 153-8914, Japan} 
\email{inasa@ms.u-tokyo.ac.jp}



\begin{document}
\maketitle

\begin{abstract}
 We study surface links whose link groups are free abelian, and construct various stimulating and highly non-trivial examples of such surface links. 
\end{abstract}

\section{Introduction}%%%%%%%%
 A {\em surface link} is the image of a smooth embedding of closed surfaces into the Euclidean 4-space $\R^{4}$ (or the 4-sphere $S^{4}$, according to the context). In this paper, we always assume that surfaces are oriented. 

The {\em (total) genus} $g(S)$ of a surface link $S$ is the sum of the genera of the components of $S$. A surface link $S$ is called a {\em 2-link} (resp. {\em $T^{2}$-link}) if each component is a sphere (resp. torus).
Basic facts on surface links can be found in \cite{CKS} which we will use as a main reference.

The {\em link group} of $S$ is the fundamental group of the complement of $S$. In this paper, we study surface links whose link groups are abelian, which will be called {\em abelian surface links}. 
If the link group of $S$ is abelian then it is free of rank $n$, so we will often call an $n$-component abelian surface link an abelian surface link of {\em rank} $n$.

An abelian surface link is the simplest in the sense that its link group is the simplest among all surface links. So one might think that an abelian surface link is not an interesting object to study. In some cases, including the classical case, it is indeed true:
The link group of a classical link $L$ is abelian if and only if $L$ is either an unknot, or a Hopf link \cite[Theorem 6.3.1 -- Exercise 6.3.3]{Kawauchi}. There are no abelian 2-links of rank greater than one \cite[Chapter 3, Corollary 2]{Hillman}; (This also follows from our Theorem \ref{theorem:bound}).

The aim of this paper is to show that there are rather huge non-trivial and stimulating abelian surface links and to demonstrate that abelian surface links are interesting objects to study, despite the first impression.

We begin with studying a constraint about genus and ranks for abelian surface links in Section \ref{sec2}. This motivates us to investigate abelian surface links with small genus and non-trivial double and triple linking numbers. 

In Section \ref{sec:const}, we give a simple construction of abelian surface links of arbitrary rank. Although this construction is less interesting since it yields abelian surface links with trivial double and triple linking numbers, it has an advantage that we get a ribbon abelian surface link. We will provide more complicated and interesting abelian surface links of relatively small total genus by using torus-covering $T^{2}$-links. 

The basics of torus-covering $T^2$-links will be reviewed in Section \ref{sec3}. In Section \ref{sec4}, we prove the double and triple linking number formulae for torus-covering $T^2$-links, which are interesting in their own right.

In Section \ref{sec5}, we construct a variety of stimulating examples of abelian surface links.
We produce abelian $T^{2}$-links of rank four having various double and triple linking numbers. We also demonstrate an example of a pair of inequivalent abelian surface links with the same double and triple linking numbers. These examples illustrate richness of abelian surface links. We also construct abelian surface links of high ranks with smaller total genus than those given in Section \ref{sec:const}.

\section{Genus-rank inequality for abelian surface links}\label{sec2}

 First of all, we study fundamental constraints for abelian surface links. The next theorem provides a lower bound of the total genus of abelian surface links. 

\begin{theorem}
\label{theorem:bound}
If $S$ is an abelian surface link of rank $n>1$, then we have an inequality
\begin{equation}
\label{eqn:bound}
n(n-1) \leq 4 g(S).
\end{equation} 
\end{theorem}
\begin{proof}
Let us put $g=g(S)$. By Alexander duality, 
\[ \left\{ \begin{array}{l}
H_{2}(S^{4}-S ;\Z) = \Z^{2g} \\
H_{3}(S^{4}-S ; \Z) = \Z^{n-1}.
\end{array}
\right.
\]
We can construct $K(\Z^{n}, 1)$ from $S^{4}- S$ by attaching cells whose dimensions are greater than two. Hence we have a surjection
\begin{equation}\label{eqn1}
H_{2}(S^{4}- S;\Z) \rightarrow H_{2}(\Z^{n};\Z) = \Z^{\binom{n}{2}} 
\end{equation}
so $\frac{n(n-1)}{2} \leq 2g$ holds.
\end{proof}

\begin{corollary}
\label{cor:bound}
The rank of an abelian $T^{2}$-link is at most $5$. 
\end{corollary}

We also observe the following simple fact.

\begin{proposition}
\label{prop:group}
If $S$ is an abelian surface link, so is its sublink.
\end{proposition}
\begin{proof}
For an $n$-component abelian surface link $S= \bigsqcup_{i=1}^{n} F_{i}$, let $N(F_{n})$ and $\mu_{n}$ be the regular neighborhood and the meridian of the $n$-th component $F_{n}$, and $S'= \bigsqcup_{i=1}^{n-1} F_i$ be a sublink of $S$. By van-Kampen's theorem, 
\[ \pi_{1}(S^{4}- S') = \pi_{1}(S^{4}-S)\slash \langle \! \langle \mu_{n} \rangle\!\rangle \]
where $\langle\!\langle \mu_{n} \rangle\!\rangle$ denotes the normal subgroup of $\pi_{1}(S^{4}-S)$ generated by $\mu_{n}$. Hence $\pi_{1}(S^{4}- S')$ is abelian if so is $\pi_{1}(S^{4}-S)$. Since $H_{1}(S^{4}-S') = \Z^{n-1}$, we conclude $\pi_{1}(S^{4}-S')= \Z^{n-1}$.   
\end{proof}

With Proposition \ref{prop:group}, Theorem \ref{theorem:bound} provides more constraints of the genus of abelian surface links.

 \begin{corollary}\label{cor:genus}
Let $S$ be an abelian surface link. For each $g\geq 0$, the number of genus $g$ components of $S$ is at most $4g+1$. 
\end{corollary}

Proposition \ref{prop:group} implies that $\pi_{1}(S^{4}-F) =\Z$ for each component $F$ of an abelian surface link $S$.
Thus according to the famous unknotting conjecture, it is expected that each component $F$ is unknotted so the link-homotopy will carry substantial information. This is why we take notice of the double and the triple linking numbers, and why we are interested in finding abelian surface links of low genus, with non-trivial double and triple linking numbers. 

It is an interesting question to ask whether the lower bound (\ref{eqn:bound}) is best-possible. The lower bound is sharp for the case $n< 4$:
\begin{enumerate}
\item Proposition \ref{prop:exist} shows that there exists a $2$-component abelian surface link of genus $1$.
\item Hopf 2-links with beads \cite{CKSS01} is a $3$-component abelian surface link of genus $2$.
\item There exists a 4-component abelian $T^{2}$-link \cite{N3}.  
\end{enumerate}

The first unknown case is $n=4$: are there $4$-component abelian surface links of genus $3$?
%As we will see later in Proposition \ref{prop:highgenus}, there exists a $5$-component abelian surface link of genus $7$.

\section{Simple constructions of abelian surface links}\label{sec:const}

We provide a simple construction of abelian surface links of arbitrary rank.

\begin{proposition}
\label{prop:exist}
For arbitrary large $n>0$, there exists an abelian surface link of rank $n$. Indeed, there exists an abelian surface link of rank $n$, of genus $\frac{n(n-1)}{2}$.
\end{proposition} 
\begin{proof}
The case $n=1$ is obvious (take an unknotted sphere). We show that from an $n$-component abelian surface link $S$, one can construct an $(n+1)$-component abelian surface link $S^{+}$ of total genus $g(S)+n$.
 
Take a surface link diagram $D$ of $S$ and a point $z \in \R^{3} -D$. For each component $F_{i}$ of $S$, take a regular point $x_{i} \in D \subset F_{i}$. By applying Roseman moves if necessary, we may assume that all points $z,x_{1},\ldots,x_{n}$ lie on the closure of a certain component of $\R^{3}-D$, say $C$. Take mutually disjoint paths $\gamma_{i} \subset C$ that connects $z$ and $x_{i}$. 
 
We add a new genus $n$ component $F_{n+1}$ to $S$ as follows. $F_{n+1}$ is made of the sphere neighborhood of $z$ and $n$ handles contained in a neighborhood of $\gamma_{i}$.
Near $x_{i}$ each handle links to the sheet of $F_{i}$ as depicted in Figure \ref{fig:construction}.

From the Wirtinger presentation, it is confirmed that $S^{+} = S \cup F_{n+1}$ is an abelian surface link of rank $n+1$.
\end{proof}

\begin{figure}[htbp]
 \begin{center}
\includegraphics*[width=70mm]{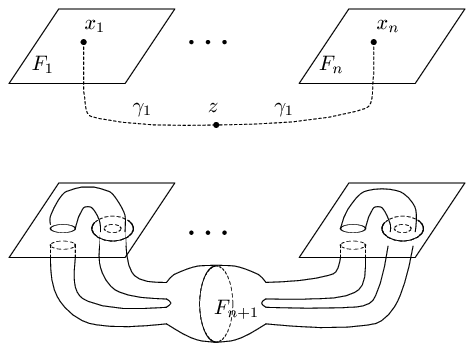}
 \caption{Construction of $S^{+}$}
 \label{fig:construction}
  \end{center}
\end{figure}

A surface link is called {\it ribbon} if it is obtained from the split union of unknotted spheres by adding a finite number of 1-handles.
In our construction above, if $S$ is ribbon, then so is $S^{+}$, hence we actually proved the following slightly stronger fact. 

\begin{corollary}
There exists a ribbon abelian surface link of rank $n$, with genus $\frac{n(n-1)}{2}$.
\end{corollary}

The newly-added $(n+1)$-st component $F_{n+1}$ links to other components of $S^{+}$ in a rather trivial way, in the sense that $\Dlk_{i,n+1}=0$ and $\Tlk_{i,n+1,j}=0$ for all $i,j \in \{1,\ldots,n\}$. See Section \ref{sec4} for the double linking number $\Dlk_{i,j}$ and the triple linking number $\Tlk_{i,j,k}$. 

\section{A torus-covering $T^2$-link and its link group} \label{sec3}
 
In this section, we review torus-covering $T^2$-links and their link groups. For details, see \cite{N}. 

Let $T$ be the standardly embedded oriented torus in $\mathbb{R}^4$, the boundary of an unknotted (standardly embedded) solid torus in $\mathbb{R}^3 \times \{0\} \subset \mathbb{R}^4$. Let $N(T)$ be a tubular neighborhood of $T$, and let $p \,:\, N(T) \to T$ be the natural projection. 

\begin{df} \label{Def2-1} 
A {\it torus-covering $T^2$-link} is a 
$T^{2}$-link $S$ in $\mathbb{R}^4$ (or $S^{4}$)
such that $S$ is contained in $N(T)$ and $p |_{S} \,:\, S \to T$ is an unbranched covering map. 
\end{df}

Let $S$ be a torus-covering $T^{2}$-link.
Fix a base point $x_0 =(x'_{0},x''_{0})$ of $T = S^{1} \times S^{1}$.
Recall that $T$ is embedded as $T=\partial (D^{2} \times S^{1}) \subset \R^{3} \times \{0\} \subset \R^{4}$. Take two simple closed curves on $T$, $\mathbf{m} = \partial D^{2} \times \{x''_0\}$ and $\mathbf{l}= \{x'_{0}\}\times S^{1}$.

Let us consider the intersections $S \cap p^{-1}(\mathbf{m}) \subset \mathbf{m} \times D^{2}$ and $S \cap p^{-1}(\mathbf{l}) \subset \mathbf{l} \times D^{2}$. They are regarded as closed $m$-braids in the 3-dimensional solid tori, where $m$ is the degree of the covering map $p|_{S} \, :\, S  \rightarrow T$. Cutting open the solid tori along the 2-disk $p^{-1}(x_0)= \{x_{0}\}\times D^{2}$, we obtain two $m$-braids $a$ and $b$.
The assumption that $p|_{S}$ is an unbranched covering implies that $a$ and $b$ commute. We call the pair of commutative braids $(a,b)$ the {\it basis braids} of $S$. Conversely, starting from a pair of commutative $m$-braids $(a,b)$, one uniquely constructs a torus-covering $T^2$-link with basis braids $(a,b)$ \cite[Lemma 2.8]{N}.
For commutative $m$-braids $a$ and $b$, we denote by $\mS_{m}(a,b)$ the torus-covering $T^2$-link with basis braids $a$ and $b$. 

For the $i$-th component $F_{i}$ of $S$, fix the lift of the base point $x_{i} \in p|_{S}^{-1}(x_{0}) \subset F_{i}$. Let $p_{i} = p|_{F_{i}} : F_{i} \rightarrow T$ be the restriction of $p$. Let $\mathbf{m}_{i}$ and $\mathbf{l}_{i}$ be the connected components of $p_{i}^{-1}(\mathbf{m})$ and $\mathbf{l}_{i}=p_{i}^{-1}(\mathbf{l})$ that contains $x_{i}$, respectively. Then two curves $\{\mathbf{m}_{i}, \mathbf{l}_{i}\}$ form an oriented basis of $H_{1}(F_{i})$ which is compatible with the orientation of $F_{i}$. We say $\{\mathbf{m}_{i}, \mathbf{l}_{i}\}$ is a {\em preferred basis} of $H_{1}(F_{i})$. A preferred basis is independent of choices of a lift of base points $x_{i}$ $(i=1,\ldots,n)$.

Let $B_{n}$ be the $n$-strand braid group and $\sigma_{1},\ldots,\sigma_{n-1}$ be the standard generator of $B_{n}$.
Let $\mathbf{F}_n$ be the free group of rank $n$ generated by $x_{1},\ldots,x_{n}$ and let $\mathcal{A}: B_{n} \rightarrow \Aut(\mathbf{F}_{n})$ be the Artin representation of the braid groups, which is given by 
\begin{gather*}
\mathcal{A}^{\sigma_j}(x_i)=\begin{cases}
                        x_{i+1} & \text{if $i = j$} \\
                        x_{i}^{-1}x_{i-1}x_{i} & \text{if $i = j+1$} \\
                        x_i & \text{otherwise.} 
\end{cases}
\end{gather*}
(Here $\mathcal{A}^{a} = \mathcal{A}(a)$ for $a \in B_n$ and we use the left action of $\Aut(\mathbf{F}_{n})$ on $\mathbf{F}_{n}$. Note that the action is inverse of the one given in \cite[Chapter 27]{Kamada02} and \cite[Section 3]{N}.) As we have shown in \cite[Proposition 3.1]{N}, the link group of $\mathcal{S}_m(a,b)$ is presented by 
\begin{equation*}
\pi_{1}(S^{4}-\mathcal{S}_m(a,b)) = \left\langle \, x_1 \,, \ldots,\, x_m \mid
x_i=\mathcal{A}^{a}(x_i)=\mathcal{A}^{b}(x_i), \;\;(i=1,2,\ldots, m ) \, \right\rangle. 
\end{equation*}

\section{Double and Triple linking numbers of torus-covering $T^{2}$-links} \label{sec4}%%%%%%%%%

The double linking numbers and the triple linking numbers are natural generalizations of the classical linking numbers, and are the most fundamental link-homotopy invariants of surface links. The double and triple linking numbers are complete invariants of {\em link bordism classes}: Two surface links are link bordant if and only if they have the same double and triple linking numbers \cite{CKSS01,San,San2}. In this section, we give formulae for the double and triple linking numbers of torus-covering $T^{2}$-links in terms of their basis braids.

\subsection{Double linking numbers}

Among various equivalent definitions of the double linking number, we use the following diagrammatic one \cite{CKS}.

Take a surface link diagram $D \subset \R^{3}$ of a surface link $S$ and let $F_i$ and $F_j$ be two different components of $S$. We say that a double point curve of $D$ is of type $(i,j)$ if on the double point curve $F_{i}$ appears as the over sheet and $F_{j}$ appears as the under sheet. Let $D_{ij}$ be the double point curves of type $(i,j)$, and let $D_{ij}^+$ be the immersed circles in $\mathbb{R}^3$ obtained by shifting $D_{ij}$ in the diagonal direction so that $D_{ij}$ and $D_{ij}^+$ are disjoint. Let $\widetilde{D}_{ij}$ and $\widetilde{D}_{ij}^+$ be a link in $\mathbb{R}^3$ obtained from $D_{ij}$ and $D_{ij}^+$ by a slight perturbation.

The {\em double linking number} $\Dlk_{i,j}=\Dlk_{i,j}(S) \in \mathbb{Z}\slash 2\mathbb{Z}$ is defined as the linking number modulo two,
\[
\Dlk_{i,j}(S)= \lk( \widetilde{D}_{ij}, \widetilde{D}_{ij}^+) \pmod{2}.
\]

To give formulae of the double and the triple linking numbers for torus-covering $T^2$-links, we introduce the following notations.

For a torus-covering $T^{2}$-link $\mS_{m}(a,b)$, let $\widetilde{A}_{i}$ be the components of the closed braid $\widehat{a}$ that corresponds to the $i$-th component of $\mS_{m}(a,b)$. In general, $\widetilde{A}_{i}$ is not connected. Take one of its connected components of $\widetilde{A}_{i}$ and call it $A_{i}$.
Then $\widetilde{A}_{i}$ and $A_{i}$ are regarded as an oriented link in $\R^{3}$.

We define $\lk^{a}_{i,j}$, the {\em linking number of the $i$-th component to the $j$-th component in the $a$-direction} by the classical linking number
\[ \lk_{i,j}^{a} = \lk ( A_{i}, \widetilde{A}_{j}). \]
The notations $\widetilde{B}_{i}$, $B_{i}$ and $\lk^{b}_{i,j}$ are defined similarly.
As we will see in Remark \ref{rem:independent}, $\lk^{a}_{i,j}$ does not depend on a choice of a connected component $A_{i}$.
We remark that $\lk^{a}_{i,j}$ is not always symmetric, and it might happen $\lk^{a}_{i,j} \neq \lk^{a}_{j,i}$.

We consider the special case that $b=\Delta^{2N}$, where $\Delta$ is an $m$-braid with a half twist and $N$ is an integer. 
In such case, $\lk^{\Delta^{2}}_{i,j}$ is equal to the degree of the covering $p|_{F_{j}}: F_{j} \rightarrow T$ corresponding to the $j$-th component, which will be denoted by $m_j$, so $\lk^{\Delta^{2N}}_{i,j}=N m_j$. Moreover, $\widetilde{A}_{i}$ is connected so $A_{i} = \widetilde{A}_{i}$ for all $i$, thus $\lk^{a}_{i,j}$ is symmetric. 

The following theorem gives the double linking number for torus-covering $T^2$-links of the form $\mS_{m}(a,\Delta^{2N})$.

\begin{theorem}\label{theorem:dlk}
Let $\mathcal{S}_m(a, \Delta^{2N})$ be a torus-covering $T^{2}$-link. Then the double linking number of the $i$-th and $j$-th components is given by the formula
\begin{equation*}
\Dlk_{i,j}(\mathcal{S}_m(a, \Delta^{2N}))=N( \lk_{i,j}^{a} + m_{i} + m_{j}) \mod{2}.
\end{equation*}
\end{theorem}

\begin{proof}
We will actually prove 
\begin{equation*}
\label{eqn:proof}
\Dlk_{i,j}(\mathcal{S}_m(a, \Delta^{2N}))=N \left( \lk_{i,j}^{a} +\frac{m_i m_j}{lcm} (lcm-1) \right) \mod{2},
\end{equation*}
where $lcm$ denotes the least common multiple of $m_i$ and $m_j$.
Since $\frac{m_i m_j}{lcm}(lcm-1) = m_{i}+ m_{j} \pmod{2}$, this proves the theorem.

\begin{figure}
\centerline{\includegraphics{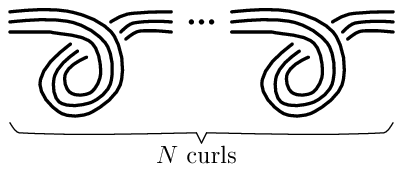}}
\caption{The braid $\Delta^{2N}$ as $N$ band-twists}
\label{fig:ncurl}
\end{figure}

Let us consider a projection $\mathbb{R}^4 \to \mathbb{R}^3$ that extends the natural projection $N(T) \cong [0,1] \times [0,1] \times T \rightarrow [0,1] \times T$, and let $D$ be the link diagram obtained from this specified projection.
Then $D$ appears as an orbit of braid diagrams transforming $a \Delta^{2N}$ to $\Delta^{2N} a$ identifying the initial and the terminal braids in an analogous way of forming a torus; see \cite[Section 2.2]{N2} for detailed arguments. For an easy treatment of double point curves, we deform the braid $\Delta^{2N}$ as $N$ band-twists (curls) described in Figure \ref{fig:ncurl}.
 
We say a crossing $c$ in the braid $a$ is of type $(i,j)$ if it is a crossing between two links $\widetilde{A_{i}}$ and $\widetilde{A_{j}}$ such that $\widetilde{A_{i}}$ appears as an over-arc. For the braid $\Delta^{2N}$, a crossing of type $(i,j)$ is defined similarly.
Then the double point curves of $D$ appear as the orbit of crossings of type $(i,j)$. We say a double point curve is of type A (resp. B) if it is an orbit of crossings of type $(i,j)$ in the braid $a$ (resp. $\Delta^{2N}$). 

First we analyze double point curves of type A. We have deformed $\Delta^{2N}$ as $N$ band-twists described in Figure \ref{fig:ncurl} so each crossing of type $(i,j)$ in $a$ forms one double point curve which is a simple closed curve with $N$ curls, when the braid $a$ slides along $\Delta^{2N}$.
The diagonal direction of the double point curve agrees with the blackboard framing, so each double point curve of type A contributes to the double linking number $\Dlk_{i,j}$ by $N$ (see Figure \ref{fig:doublecurve} (A)). The number of crossings of type $(i,j)$ in $a$ is $\lk_{i,j}^{a}$, so the double point curves of type A contribute to $\Dlk_{i,j}$ by $N \lk_{i,j}^{a} \pmod{2}$. 

To treat double point curves of Type B, we look at the crossings of type $(i,j)$ in $\Delta^{2N}$. Since $\widetilde{B}_{i}$ is an $(m_{i},Nm_{i})$-torus link, the number of crossings of type $(i,j)$ is $N m_i m_j$.
For each $k$ and $l$ $(1\leq k, l \leq m)$, there are $N$ crossings of type $(i,j)$ formed by the $k$-th component of $\widetilde{B}_{i}$ and the $l$-th component of $\widetilde{B}_{j}$. We denote these crossings by $c_{kl}^{n}$ $(n=1,\ldots,N)$. 

The braid $a$ induces the cyclic permutation of $m_{i}$ (resp. $m_{j}$) elements consisting of the connected components of $\widetilde{B}_{i}$ (resp. $\widetilde{B}_{j}$). We denote these cyclic permutations by $\sigma$ and $\tau$, respectively.
 
When the braid $a$ slides through $\Delta^{2N}$, the crossing $c^{n}_{kl}$ moves to $c^{n}_{\sigma(k) \tau (l)}$ for each $n$. Since $\sigma$ and $\tau$ are cyclic permutations, each double point curve of type B rounds $lcm$ times (see Figure \ref{fig:doublecurve} (B--1)) and the number of the double point curves of type $B$ is $N m_{i}m_{j}/lcm$.
 
By slight perturbation, each double point curve of type $B$ is modified as the $(lcm, 1)$-curve on a standardly embedded torus so that the diagonal direction agrees with the outward-normal direction of the torus (see Figure \ref{fig:doublecurve} (B--2)).
Hence each component of the double point curves of type B contributes to $\Dlk_{i,j}$ by $(lcm-1)$. 

Since there are $Nm_i m_j/lcm$ double point curves of type B, we conclude that the double point curves of type B contribute to $\Dlk_{i,j}$ by $N m_{i}m_{j}(lcm-1)/lcm$.

\begin{figure}
\centerline{\includegraphics{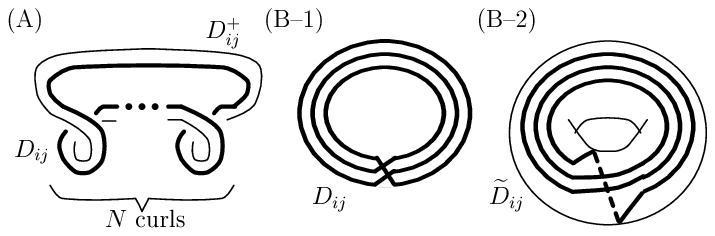}}
\caption{Double point curves}
\label{fig:doublecurve}
\end{figure}
\end{proof}

\begin{rem}
A similar argument shows that 
\[ \Dlk_{i,j}(\mS_{m}(a,a^{N})) = N \pmod{2}. \]
\end{rem}

\subsection{Triple linking numbers}

In contrast with our treatment of the double linking number, 
we use the following algebraic-topological definition of the triple linking number.

For three distinct components $F_{i}, F_{j}$ and $F_{k}$ of a surface link $S$, the triple linking number $\Tlk_{i,j,k} \in \Z$ is defined as the framed intersection 
\[ \Tlk_{i,j,k} = M_i \cdot F_j \cdot M_k \in \pi_{4}(S^{4})=H_{4}(S^{4})=\Z \]
 where $M_x$ $(x=i,k)$ is a Seifert hypersurface of $F_{x}$.
Thus $\Tlk_{i,j,k}$ is the algebraic intersection number of two curves on $F_{j}$, $F_{j} \cap M_i$ and $F_{j} \cap M_k$. 
Here we remark that in a similar manner, we are able to define the triple linking number for the case $F_{i}=F_{k}$, but in this case $\Tlk_{i,j,k}=0$. Hence we always treat triple linking numbers of mutually distinct three components.
 
\begin{proposition}
\label{prop:tlkform}
Let $\iota: F_{j} \hookrightarrow S^{4}-(F_{i} \cup F_{k})$ be the inclusion map.
If $F_{j}$ is a torus, then 
\[ \Tlk_{i,j,k} = \mathrm{det}( \iota_{*}: H_{1}(F_{j}) \rightarrow H_{1}(S^{4}-(F_{i} \cup F_{k})) ).\]
\end{proposition}
\begin{proof}
For $x=i,k$, the cohomology class $[F_{j} \cap M_x] \in H^{1}(F_{j})$ represented by the curve $F_{j} \cap M_{x}$ is given by 
\[ [F_{j} \cap M_x] = \iota_{*}([F_{x}]) \in H_{2}(S^{4}-F_{j}) \cong H^{1}(F_{j}). \]
Since the algebraic intersection number of curves on a torus is given by the determinant, 
\[ \Tlk_{i,j,k} = \mathrm{det}( \iota_{*}: H_{2}(F_{i} \sqcup F_{k}) \rightarrow H_{2}(S^{4}-F_{j}) ). \] 
Since $H_{2}(F_{i} \sqcup F_{k}) \cong H^{1}(S^{4}-(F_{i}\cup F_{k}))=\textrm{Hom}(H_{1}(S^{4}- (F_{i}\cup F_{k})),\Z)$ and $H_{2}(S^4-F_j) \cong H^{1}(F_j)=\textrm{Hom}(H_{1}(F_j),\Z)$, we conclude 
\[ \Tlk_{i,j,k} = \mathrm{det} ( \iota_{*}: H_{1}(F_{j}) \rightarrow H_{1}(S^{4}-(F_{i}\cup F_{k}))  ). \]
\end{proof}

This leads to a simple formula of the triple linking numbers for a torus covering $T^{2}$-link $\mS_{m}(a,b)$.

\begin{theorem}\label{theorem:tlk}
For a torus-covering $T^{2}$-link, the triple linking number of the $i$-th, $j$-th and $k$-th components ($i,j,k$ are distinct) are given by the formula
\[ \Tlk_{i,j,k}(\mS_{m}(a,b)) = \lk^{a}_{j,i}\lk^{b}_{j,k} -  \lk^{a}_{j,k}\lk^{b}_{j,i}.\]
\end{theorem}

\begin{proof}
Let $\iota:F_{j} \hookrightarrow S^{4}-(F_{i} \cup F_{k})$ be the inclusion map.
The first homology group $H_{1}(S^{4}- (F_{i} \cup F_{k}) )$ is generated by the meridians of $F_{i}$ and $F_{k}$ which will be denoted by $\mu_{i}$ and $\mu_{k}$. Let $\{\mathbf{m}_{j},\mathbf{l}_{j}\}$ be a preferred basis of $H_{1}(F_{j})$.
Recall that the closed braid $\widehat{a}$ appears as $S \cap p^{-1}(\mathbf{m}) \subset \mathbf{m} \times D^{2}$, where $S=\mathcal{S}_m(a,b)$. So $\widetilde{A}_{i}$ is nothing but $F_{i} \cap p^{-1}(\mathbf{m}) \subset \mathbf{m} \times D^{2}$.

Now choose $A_{j}$ as the component of $\widehat{a}$ that contains $x_{j}$ (here $x_{j}$ is a fixed lift of the base point $x_{0}$ as we defined in Section \ref{sec3}). By definition, $A_{j}$ is regarded as an curve on $F_{j}$.

On the other hand, a preferred basis $\mathbf{m}_{j}$ was defined as a connected component of $p|_{F_{j}}^{-1} (\mathbf{m}) \subset F_{j}$ that contains the lift of the base point $x_{j} \in F_{i}$. Hence as a curve on $F_{j}$, $A_{j} = \mathbf{m}_{j} \in H_{1}(F_{j})$.

This concludes 
\[
\iota_{*}(\mathbf{m}_{j}) = \lk^{a}_{j,i} [\mu_{i}] + \lk^{a}_{j,k} [\mu_{k}]. 
\]
By similar arguments for $\mathbf{l}$, we get
\[ \iota_{*}(\mathbf{l}_{j}) = \lk^{b}_{j,i} [\mu_{i}] + \lk^{b}_{j,k} [\mu_{k}]. \]
By Proposition \ref{prop:tlkform} we get the desired formula.
\end{proof}

\begin{rem}\label{rem:independent}
The proof of Theorem \ref{theorem:tlk} implies that 
$\lk^{a}_{j,i}$ (and $\lk^{b}_{j,i}$) is determined by the $j$-th peripheral subgroup, that is,
the image of $\iota_{*}: \pi_{1}(F_{j}) \rightarrow \pi_{1}(S^{4}-S)$, where $\iota: F_{j} \hookrightarrow S^{4}-S$ is the inclusion map.
In fact, the observation that $A_{j} = \mathbf{m}_{j} \in H_{1}(F_{j})$ implies that 
\[ \iota_{*}(\mathbf{m}_{j}) = \sum_{i=1}^{n} \lk_{j,i}^a[\mu_{i}] \in H_{1}(S^{4}-S), \]
where $[\mu_{i}]$ denotes the meridian of $F_{i}$ and $\lk^{a}_{j,j}$ is treated as zero. In particular, $\lk^{a}_{j,i}$ does not depend on a choice of the connected component $A_{j}$.
\end{rem}

\begin{corollary}
\label{cor:tlk}
We have 
\[\Tlk_{i,j,k}(\mathcal{S}_m(a,\Delta^{2N})) = N (m_{k}\lk^{a}_{j,i} -  m_{i}\lk^{a}_{j,k}) \]
where $m_{i}$ denotes the degree of the unbranched covering $p|_{F_{i}}: F_{i} \rightarrow T$ corresponding to the $i$-th component.
\end{corollary}

\begin{rem}
Theorem \ref{theorem:tlk} is a generalization of \cite[Theorem 1.1]{N2}, where both $a$ and $b$ are assumed to be pure braids. Along the same lines of the proof of \cite[Theorem 1.1]{N2} and Theorem \ref{theorem:dlk}, a similar diagrammatic argument that regards the triple linking number as the algebraic count of triple points of type $(i,j,k)$ provides an alternative proof of Theorem \ref{theorem:tlk}; we remark that since we adopt a different definition of the triple linking number, the sign of each triple linking number would be reversed. 
\end{rem}

\section{Examples of abelian surface links} \label{sec5}%%%%%%%%

\subsection{Rank four examples}

Let 
\[ H_{m} = \langle a_{1},\ldots, a_{m} \: | \: a_{i}(a_{1}a_{2}\cdots a_{m}) = (a_{1}a_{2}\cdots a_{m}) a_{i}\; (i=1,\ldots, m) \rangle \]
 be the link group of the $(m,m)$-torus link, and let $\pi:\mathbf{F}_{m}=\langle x_{1},\ldots,x_{m}\rangle \rightarrow H_{m}$ be the surjection defined by $\pi(x_{i}) = a_{i}$. 

Since $\mathcal{A}^{a}(x_{1}x_{2}\cdots x_{m}) = x_{1}x_{2}\cdots x_{m}$ for any $m$-braid $a$, $\mathcal{A}^{a}$ induces an automorphism of $\pi(\mathbf{F}_m)=H_m$, $\tmA^{a}: H_m \rightarrow H_m$.
Thus the link group of $\mathcal{S}_{m}(a,\Delta^{2})$ is expressed as the quotient of $H_{m}$ by the set of the relations 
\[ 
\pi_1(\mathbb{R}^4-\mathcal{S}_m(a,\Delta^{2}))= H_{m}\slash \langle a_{i} = \tmA^{a}(a_{i}) \rangle. \]

This makes some calculations simple.
For example, for the full-twist of first $(m-1)$-strands $\Delta'^{2} = (\sigma_{1}\sigma_{2}\cdots \sigma_{m-2})^{m-1} \in B_{m}$ we have 
\begin{equation}
\label{eqn:reduce}
 \tmA^{\Delta'^{2}}(a_{i}) = a_{m}a_{i}a_{m}^{-1}. 
\end{equation}

For $k,l \in \Z_{\geq 1}$ and $\be =(e_{1},e_{2},e_{3}) \in \{\pm 1\}^{3}$, 
let us define $X_{k,l,\be}, Y_{k,l,\be} \in B_{k+l+2}$ and $Z_{k,\be} \in B_{k+3}$ as
\[
\left\{
\begin{array}{l} X_{k,l,\be}= \sigma_{1}^{2 e_{1}} (\sigma_{2}\sigma_{3}\cdots\sigma_{k}) \sigma_{k+1}^{2e_{2}} (\sigma_{k+2}\sigma_{k+3}\cdots\sigma_{k+l}) \sigma_{k+l+1}^{2e_{3}}\\
Y_{k,l,\be}
 = (\sigma_{1}\sigma_{2}\cdots\sigma_{k+l})^{(k+l+1) e_{1}} (\sigma_{k+l+1}\sigma_{k+l} \cdots \sigma_{k+3})\sigma_{k+2}^{2e_{2}} (\sigma_{k+1}\sigma_{k-1}\cdots \sigma_{3}) \sigma_{2}^{2e_{3}}\\
 Z_{k,\be} = (\sigma_{1}\sigma_{2}\cdots\sigma_{k+1})^{(k+2) e_{1}} (\sigma_{2}\sigma_{3}\cdots\sigma_{k+2})^{(k+2) e_{2}} (\sigma_{k+2}\sigma_{k+1}\cdots \sigma_{4}) \sigma_{3} ^{2 e_{3}}.
\end{array}
\right.
\]

\begin{figure}[htbp]
\centerline{\includegraphics[width=110mm]{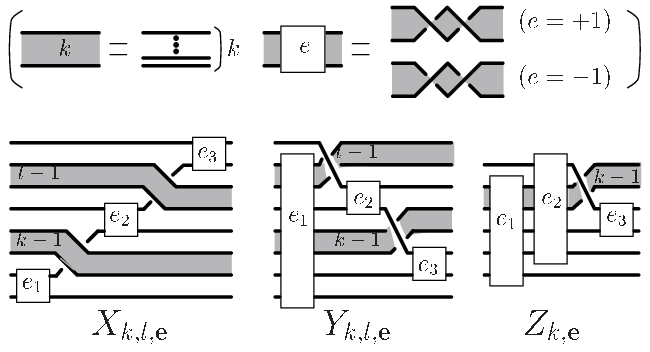}}
\caption{Braids $X_{k,l,\be}$, $Y_{k,l,\be}$, and $Z_{k,\be}$}
\label{fig:braid}
\end{figure}

\begin{theorem}
\label{theorem:T4}
Torus-covering links $\mS_{k+l+2}(X_{k,l,\be},\Delta^{2})$, $\mS_{k+l+2}( Y_{k,l,\be},\Delta^{2})$, and $\mS_{k+3}(Z_{k,\be},\Delta^{2})$ are abelian $T^{2}$-links. Their double and triple linking numbers are given by Table \ref{table:tlk}.
\end{theorem}

\begin{table}[htbp]
\caption{Double and Triple linking numbers of $\mS_{m}(a, \Delta^2)$}
\label{table:tlk}
\begin{tabular} {ccccccc}
\hline
 $a$ & $\Dlk_{1,2}$ & $\Dlk_{1,3}$ & $\Dlk_{1,4}$ & $\Dlk_{2,3}$ & $\Dlk_{2,4}$ & $\Dlk_{3,4}$ \\ \hline
$X_{k,l,\be}$ & $k$ & $l+1$ & $0$ & $k+l+1$ & $k+1$ & $l$  \\
$Y_{k,l,\be}$ & $1$ & $1$ & $0$ & $0$ & $0$  & $kl+l+1$\\
$Z_{k,\be}$ & $1$ & $1$ & $0$ & $0$ & $k$ &  $\begin{cases} k+1 \ \text{(if $k>1$)} \\
 1 \ \text{(if $k=1$)}\end{cases}$ \\ 
\hline
\end{tabular}

\vspace{0.4cm}

\begin{tabular} {ccc}
\hline
$a$  & $(\Tlk_{1,2,3}, \Tlk_{2,3,1}, \Tlk_{3,1,2})$ & $(\Tlk_{1,2,4}, \Tlk_{2,4,1}, \Tlk_{4,1,2})$  \\ \hline
  $X_{k,l,\be}$ & $(l e_{1}-e_{2}, e_{2}, -l e_{1})$ & $(e_{1},0,-e_{1})$\\
$Y_{k,l,\be}$ & $(-e_{2}, e_{2},0)$ & $(e_{1}, 0, -e_{1})$ \\
$Z_{k,\be}$ & $(-e_{2}, e_{2},0)$ & $(e_{1}-k e_{2}, ke_{2}, -e_{1})$ \\ \hline
\end{tabular}

\vspace{0.4cm}
\begin{tabular} {cccc}
\hline
$a$  & $(\Tlk_{1,3,4},\Tlk_{3,4,1}, \Tlk_{4,1,3})$ & $(\Tlk_{2,3,4}, \Tlk_{3,4,2}, \Tlk_{4,2,3})$ \\ \hline
  $X_{k,l,\be}$ & $(-e_{3},e_{3},0)$ & $(e_{2}-k e_{3}, k e_{3}, -e_{2})$\\
$Y_{k,l,\be}$ & $(k e_{1} -e_{2}, e_{2}, -k e_{1})$ & $(k e_{1}-e_2+l e_3,e_{2},-k e_{1}-l e_{3})$\\
$Z_{k>1,\be}$ & $(e_{1} -k e_{2}-e_{3}, k e_{2}+e_{3},-e_{1})$ & $(e_{1}-e_{3}, e_{3}, -e_{1})$ \\
$Z_{1,\be}$ & $(e_{1}-e_{3}, e_{3},-e_{1})$ & $(e_{1}+e_2-e_{3}, e_3-e_2, -e_{1})$ \\ \hline
\end{tabular}
\end{table}

\begin{proof}
Let $X=X_{k,l,\be}$ and let $G_{X}$ be the link group of $\mS_{k+l+2}(X ,\Delta^{2})$. Since 
\[ \tmA^{X}(a_{1})= \tmA^{\sigma_{1}^{2e_{1}}}(a_{1}) = (a_{1}a_{2})^{-e_{1}}a_{1}(a_{1}a_{2})^{e_{1}}, \]
the defining relation $\tmA^{X}(a_{1})=a_{1}$ is equivalent to the relation
\begin{equation}
\label{eqn:X1}
 a_{1}a_{2}=a_{2}a_{1}.
\end{equation}
For $i=2,\ldots,k$,
\[ \tmA^{X}(a_{i}) = a_{i+1}, 
 \]
hence the defining relations $\tmA^{X}(a_{i})=a_{i}$ $(i=2,\ldots,k)$ are equivalent to the relations
\begin{equation}
\label{eqn:X2}
a_{2}=a_{3} = \cdots = a_{k+1}.
\end{equation}
Next we observe that 
\[ \tmA^{X}(a_{k+1}) = (\tmA^{\sigma_{1}^{2e_{1}}} (a_{2})a_{k+2})^{-e_{2}}\cdot \tmA^{\sigma_{1}^{2e_{1}}} (a_{2}) \cdot (\tmA^{\sigma_{1}^{2e_{1}}} (a_{2})a_{k+2})^{e_{2}}.\]
Since $\tmA^{\sigma_{1}^{2e_{1}}} (a_{2})=a_2$ by (\ref{eqn:X1}), together with (\ref{eqn:X2}), the defining relation $\tmA^{X}(a_{k+1})=a_{k+1}$ is equivalent to the relation
\begin{equation}
\label{eqn:X3}
a_{2}a_{k+2}=a_{k+2}a_{2}.
\end{equation}
By a similar argument, using the relation (\ref{eqn:X3}) for $i=k+2,\ldots,k+l$, we conclude that the defining relation $\tmA^{X}(a_{i})=a_{i}$ of $G_{X}$ is equivalent to the relations
\begin{equation}
\label{eqn:X4}
a_{k+2}=a_{k+3}= \cdots = a_{k+l+1},
\end{equation}
and that the defining relation $\tmA^{X}(a_{k+l+2})=a_{k+l+2}$ is  equivalent to the relation
\begin{equation}
\label{eqn:X5}
a_{k+2} a_{k+l+2} = a_{k+l+2} a_{k+2}. 
\end{equation}
 
Now the relations (\ref{eqn:X1})--(\ref{eqn:X5}) and the defining relations of $H_{k+l+2}$ show that $G_{X}$ is a free abelian group of rank four generated by
\[ a_{1}, a_{2} (=a_{3} = \cdots = a_{k+1}), a_{k+2}( =a_{k+3}= \cdots = a_{k+l+1}), a_{k+l+2}. \]

Next let $Y=Y_{k,l,\be}$ and let $G_{Y}$ be the link group of $\mS_{k+l+2}(Y ,\Delta^{2})$. The proof that $G_{Y} \cong \Z^{4}$ is similar to the case $G_{X}$. By the relation (\ref{eqn:reduce}), 
\[ \tmA^{Y}(a_{1})= a_{k+l+2}^{e_{1}}a_{1}a_{k+l+2}^{-e_{1}}, \]
so the relation $\tmA^{Y}(a_{1})=a_{1}$ is equivalent to the relation
\begin{equation}
\label{eqn:Y1}
 a_{1}a_{k+l+2}=a_{k+l+2}a_{1}.
\end{equation}
For $i= k+4,\ldots, k+l+2$, it follows from the relation (\ref{eqn:reduce}) that 
\[ \tmA^{Y}(a_{i}) = a_{k+l+2}^{-(1-e_{1})}a_{i-1}a_{k+l+2}^{1-e_1}, \]
so the relations $\tmA^{Y}(a_{i})=a_{i}$ $(i=k+4,\ldots,k+l+2)$ are equivalent to the relations 
\begin{equation}
\label{eqn:Y2}
 a_{k+3}= a_{k+4} = \cdots = a_{k+l+1}= a_{k+l+2}.
\end{equation}
Then, by the relations (\ref{eqn:reduce}) and (\ref{eqn:Y2}), it turns out the relation $\tmA^{Y}(a_{k+3}) = a_{k+3}$ is equivalent to the relation
\begin{equation}
\label{eqn:Y3}
a_{k+2}a_{k+3} = a_{k+3}a_{k+2}, 
\end{equation}
and the relations (\ref{eqn:reduce}) and (\ref{eqn:Y3}) imply that the relations $\tmA^{Y}(a_{i}) = a_{i}$ $(i=4,\ldots, k+2)$ are equivalent to the relations
\begin{equation}
\label{eqn:Y4}
 a_{3}= a_{4}= \cdots = a_{k+2}.
\end{equation}
Finally, the relation $\tmA^{Y}(a_{3}) = a_{3}$ leads to the relation
\begin{equation}
\label{eqn:Y5}
 a_{2}a_{3}= a_{3}a_{2}.
\end{equation}
The relations (\ref{eqn:Y1})--(\ref{eqn:Y5}) and the defining relations of $H_{k+l+2}$ imply that $G_{Y}$ is a free abelian group of rank four generated by
\[ a_{1}, a_{2}, a_{3}( = a_{4}= \cdots = a_{k+2}) , a_{k+3} = (a_{k+4}= \cdots = a_{k+l+2}). \]
 
The assertion for $\mS_{k+3}(Z_{k,\be},\Delta^{2})$ is proved in a similar way. The computations of the double and triple linking numbers are direct by using Theorems \ref{theorem:dlk} and \ref{theorem:tlk}.   
\end{proof}

\subsection{Remarks on examples of rank four}

\subsubsection{Triple point numbers}

The {\it triple point number} of a surface link $S$, denoted by $t(S)$, is the minimal number of triple points among all surface link diagrams of $S$. It is regarded as a one of generalizations of the crossing number of classical links. Since the triple linking number can be defined as signed counts of triple points (see \cite{CKS,CKSS01}), there is a lower bound of the triple point number $t(S) \geq \sum_{i\neq j, j \neq k}\left| \Tlk_{i,j,k} \right|$. 

Along the same line of the arguments in \cite[Section 3]{N3}, using charts over a torus we can determine the triple point numbers for a few cases in our examples, by constructing surface link diagrams that attain the lower bounds explicitly. See Table \ref{table:triple}. 

\begin{table}[htbp]
\caption{Triple point numbers}
\label{table:triple}
 \begin{tabular} {cc}
\hline
 $a$ & $t(\mathcal{S}_4(a, \Delta^2))$ \\ \hline
$X_{1,1,(1,1,1)}$ & $16$ \\ 
$X_{1,1,(1,1,-1)}, X_{1,1,(-1,1,1)}$ & $20$\\ 
$X_{1,1,\pm(1,-1,1)}$ & $24$\\
$Y_{1,1,\pm(1,-1,1)}$ & $28$\\
$Z_{1,\pm(1,-1,-1)}$ & $32$\\ \hline
\end{tabular}
\end{table}

A surface link diagram with the minimum number of triple points is described by a graph called a chart over the torus $T$: Recall that in the proof of Theorem \ref{theorem:dlk} we took a surface link diagram using a specified projection $\R^{4} \rightarrow \R^{3}$ that extends the natural projection $N(T) = [0,1] \times [0,1] \times T \rightarrow [0,1] \times T$. A chart over $T$ is the graph on $T$ with additional information indicated, obtained as the image of the singular loci of the surface link diagram under the further projection $[0,1] \times T \rightarrow T$. In particular, a white vertex (6-valent vertex) of a chart represents a triple point in the corresponding surface link diagram. For details, see \cite{Kamada02, N, N3}. 

Figure \ref{fig:chart} presents an example of a chart that attains the minimum triple point number, which describes a surface link diagram of $\mS_{4}(X_{1,1,(1,1,1)},\Delta^{2})$. 

\begin{figure}[htbp]
\centerline{\includegraphics{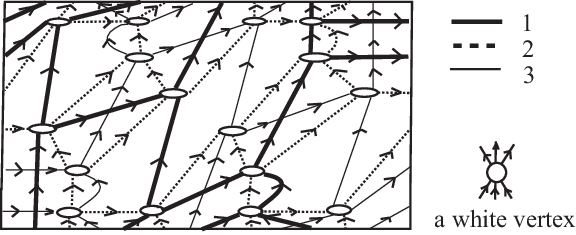}}
\caption{A $4$-chart on $T$ representing a surface link diagram of $\mathcal{S}_4(X_{1,1,(1,1,1)}, \Delta^2)=\mathcal{S}_4(\sigma_1^2\sigma_2^2\sigma_3^2, (\sigma_1\sigma_2\sigma_3)^4)$ attaining the minimum triple point numbers}
\label{fig:chart}
\end{figure}

\subsubsection{More rank four examples}

By modifying our examples we are able to construct more examples of abelian torus-covering $T^{2}$-links of rank four.

Clearly, if $\mS_{m}(a,\Delta^{2})$ is an abelian surface link, then so is 
$\mS_{m}(a\Delta^{2N},\Delta^{2})$. These two surface links have the same triple linking numbers, but may have different double linking number. 
In particular, the surface links obtained from those given in Theorem \ref{theorem:T4} by changing $a$ to $a\Delta^2$ are also abelian ($a=X_{k,l,\mathbf{e}}, Y_{k,l,\mathbf{e}}$, or $Z_{k,\mathbf{e}}$). 
By Theorem \ref{theorem:dlk}, 
\[ \Dlk_{i,j}(\mathcal{S}_m(a\Delta^{2}, \Delta^{2}))=\lk_{i,j}^{a} + m_{i} + m_{j} +m_i m_j \mod{2}, 
\]
so their double linking numbers are given by Table \ref{table3}. 

\begin{table}[htbp]
\caption{Double linking numbers of $\mathcal{S}_m(a\Delta^2, \Delta^2)$}
\label{table3}
  \begin{tabular} {ccccccc}
\hline
$a$ & $\Dlk_{1,2}$ & $\Dlk_{1,3}$ & $\Dlk_{1,4}$ & $\Dlk_{2,3}$ & $\Dlk_{2,4}$ & $\Dlk_{3,4}$ \\ \hline
$X_{k,l,\be}$ & $0$ & $1$ & $1$ & $(k+1)(l+1)$ & $1$ & $0$ \\
$Y_{k,l,\be}$ & $0$ & $k+1$ & $l$ &  $k$ & $l$  & $l+1$ \\
$Z_{k,\be}$ & $0$ & $0$ & $k$ & $1$ & $0$ & $\begin{cases} 1\  \text{(if $k>1$)} \\
 0 \ \text{(if $k=1$)}\end{cases}$ \\ \hline
 \end{tabular}
\end{table}

Let $\beta_{1} \in B_{k}$ and $\beta_{2} \in B_{l}$ be braids whose closures are unknots, and let $\iota_1: B_{k} \hookrightarrow B_{k+l+2}$ and $\iota_2:B_{l} \hookrightarrow B_{l+k+2}$ be inclusions defined by $\iota_{1}(\sigma_{i})= \sigma_{i+1}$ and $\iota_{2}(\sigma_{i})=\sigma_{i+k+1}$, respectively.
By similar calculation in the proof of Theorem \ref{theorem:T4}, for the braid 
\[ X_{\beta_{1},\beta_{2},\be} = \sigma_{1}^{2e_{1}} \iota_{1}(\beta_{1}) \sigma_{k+1}^{2e_{2}} \iota_{2}(\beta_{2}) \sigma_{k+l+1}^{2e_{3}} \in B_{k+l+2}, \]
$\mS_{k+l+2}(X_{\beta_{1},\beta_{2},\be},\Delta^{2})$ is an abelian surface link.

However, at this moment we cannot determine whether $\mS_{k+l+2}(X_{k,l,\be},\Delta^{2})$ and $\mS_{k+l+2}(X_{\beta_{1},\beta_{2},\be},\Delta^{2})$ are different or not.
In particular they have the same peripheral structures so their double and the triple linking numbers coincide.

\subsection{The double and triple linking numbers do not determine abelian $T^2$-links}

For classical links, the linking number is the complete invariant of links whose link groups are abelian: There are two links whose link groups are $\Z^{2}$, the positive and negative Hopf links, and they are distinguished by the linking number.

This is not the case for abelian surface links: we give an example of two different abelian $T^{2}$-links with the same double and triple linking numbers. 

Let $k>1$ be an odd integer and define $P_{k} \in B_{k+2}$ and $Q_{k} \in B_{3}$ by
\[
\left\{ \begin{array}{l}
 P_{k}= \sigma_{1}\sigma_{2}\cdots\sigma_{k-1} \sigma_{k}^{2}\sigma_{k+1}^{2} \\
Q_{k} = \sigma_{1}^{2}\sigma_{2}^{2k}. 
\end{array}
\right.\]

It is directly checked that $\mS_{k+3}(P_{k},\Delta^{2})$ and $\mS_{3}(Q_{k},\Delta^{2})$ are 3-component abelian $T^{2}$-links. Let $F^{P}_{1}, F^{P}_{2}$ and $F^{P}_{3}$ be the connected components of $\mS_{k+2}(P_{k},\Delta^{2})$ that corresponds to the first, the $(k+1)$-st and the $(k+2)$-nd strands of $P_{k}$, respectively. Similarly, let $F^{Q}_{1}, F^{Q}_{2}$ and $F^{Q}_{3}$ be the connected components of $\mS_{3}(Q_{k},\Delta^{2})$ that corresponds to the first, the second, and the third strands of $Q_{k}$, respectively.
 
\begin{proposition}\label{prop:same-lk}
Two abelian $T^{2}$-links $\mS_{k+2}(P_{k},\Delta^{2})$ and $\mS_{3}(Q_{k},\Delta^{2})$ have the same double and triple linking numbers, (hence they are link-cobordant), but they are not equivalent.
\end{proposition}
\begin{proof}
It is routine to check $\mS_{k+2}(P_{k},\Delta^{2})$ and $\mS_{3}(Q_{k},\Delta^{2})$ have the same double and triple linking number: By Theorem \ref{theorem:tlk}  
\[ \left\{ \begin{array}{l}
\Tlk_{1,2,3}(\mS_{k+2}(P_{k},\Delta^{2})) = \Tlk_{1,2,3}(\mS_{3}(Q_{k},\Delta^{2})) = 1-k \\
\Tlk_{2,3,1}(\mS_{k+2}(P_{k},\Delta^{2})) =\Tlk_{2,3,1}(\mS_{3}(Q_{k},\Delta^{2})) = k \\
\Tlk_{3,1,2}(\mS_{k+2}(P_{k},\Delta^{2})) = \Tlk_{3,1,2}(\mS_{3}(Q_{k},\Delta^{2})) = -1 
\end{array}
\right.
\]
and since $k$ is chosen to be odd, by Theorem \ref{theorem:dlk},
\[ 
\left\{ \begin{array}{l}
 \Dlk_{1,2}(\mS_{k+2}(P_{k},\Delta^{2})) =  \Dlk_{1,2}((\mS_{3}(Q_{k},\Delta^{2}))) = 1\\
 \Dlk_{2,3}(\mS_{k+2}(P_{k},\Delta^{2})) = \Dlk_{2,3}(\mS_{3}(Q_{k},\Delta^{2})) = 1\\
  \Dlk_{1,3}(\mS_{k+2}(P_{k},\Delta^{2})) = \Dlk_{1,3}((\mS_{3}(Q_{k},\Delta^{2}))) = 0.
 \end{array}
 \right.
 \]
 
Assume that $\mS_{k+3}(P_{k},\Delta^{2})$ and $\mS_{3}(Q_{k},\Delta^{2})$ are equivalent. Then the computation of the triple linking numbers imply that the $i$-th $(i=1,2,3)$ component of $\mS_{k+2}(P_{k},\Delta^{2})$ must correspond to the $i$-th component of $\mS_{3}(Q_{k},\Delta^{2})$.

Let $\{\mathbf{m}^{X}_3, \mathbf{l}^{X}_{3}\}$ be a preferred basis of $H_{1}(F^{X}_{3})$ and let $\iota^{X}: F^{X}_{3} \rightarrow S^{4}-(F^{X}_{1} \cup F^{X}_{2}) $ be the embedding $(X= P,Q)$. 
Then, as we have seen in the proof of Theorem \ref{theorem:tlk},
\[ \left\{\begin{array}{l}
 \iota^{P}_{*}(\mathbf{m}^{P}_{3}) = [\mu_{2}]\\
 \iota^{P}_{*}(\mathbf{l}^{P}_{3}) = k[\mu_{1}] + [\mu_{2}]
\end{array}
\right.
\] 
and 
\[ \left\{\begin{array}{l}
 \iota^{Q}_{*}(\mathbf{m}^{Q}_{3}) = k[\mu_{2}]\\
 \iota^{Q}_{*}(\mathbf{l}^{Q}_{3}) = [\mu_{1}] + [\mu_{2}]
\end{array}
\right.
\]
where $[\mu_{i}]$ denotes the meridian of the $i$-th component.  
Since $k>1$, this implies that $[\mu_{2}]$ is in $\iota^{P}_{*}(\pi_{1}(F_{3}))$ but not in $\iota^{Q}_{*}(\pi_{1}(F_{2}))$. This is a contradiction.
\end{proof}

\subsection{Some higher rank examples}
We close the paper by constructing abelian surface links with relatively small genus.

\begin{proposition}
\label{prop:highgenus}
For $n>4$, there exists an abelian surface link $S_{n}$ of rank $n$ with genus
$\frac{1}{2}(n^{2}-3n+4)$.
\end{proposition}
\begin{proof}
Let $\mS_{n}=\mS_{n}(\sigma_{1}^{2}\sigma_{2}^{2}\cdots \sigma_{n-1}^{2},\Delta^{2})$ and let $F_{i}$ be the $i$-th component of $\mS_{n}$ that corresponds to the $i$-th strand of the basis braids.

Then the link group of $\mS_{n}$ is 
\[ 
\left \langle
 x_{1},\ldots,x_{n} \: 
\begin{array}{|cc}
 x_{i}(x_{1}x_{2}\cdots x_{n}) = (x_{1}x_{2}\cdots x_{n})x_{i} & (i=1,\ldots, n)\\
 x_{i}x_{i+1}=x_{i+1}x_{i} & (i=1,\ldots, n-1)
 \end{array}
\right \rangle.
\]
Here $x_{i}$ corresponds to a meridian of $F_{i}$.

We use the following operation which is a variant of construction in Proposition \ref{prop:exist}. 
Take a surface link diagram $D$ of $\mS_{n}$. In the Wirtinger presentation of the link group, each sheet of $D$ represents certain generator.
For distinct components $F_{i}$ and $F_{j}$, take two sheets of $D$, $E_i$ and $E_j$ so that their corresponding elements are $x_{i}$ and $x_{j}$, respectively. 

By performing Roseman moves if necessary, we may assume that $E_{i}$ and $E_{j}$ are the boundaries of a certain complementary region $C$ of $\R^{3}-D$.
In the neighborhood of $C$, we add a 1-handle to $F_{i}$ to modify the diagram $D$ as shown in Figure \ref{fig:alink}. 
We say this operation an {\em addition of a 1-handle abelian linking} between $x_{i}$ and $x_{j}$, since this operation adds the relation $x_{i}x_{j}=x_{j}x_{i}$ to the link group.  

\begin{figure}[htbp]
 \begin{center}
\includegraphics*[scale=0.5, width=70mm]{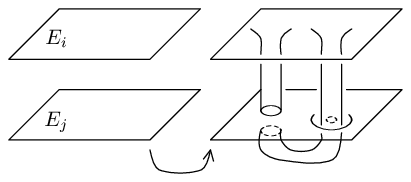}
\caption{Adding 1-handle abelian linking}
\label{fig:alink}
  \end{center}
\end{figure}

Now we are ready to construct desired abelian surface links.
First of all, we add a $1$-handle abelian linking between $x_{1}$ and $x_{n}$ to get an $n$-component surface link $\widehat{S}_{n}$ with genus $(n+1)$.
Let $G_{n}$ be the link group of $\widehat{S}_{n}$, which is presented by
\[ 
\left \langle
 x_{1},\ldots,x_{n} \: 
\begin{array}{|cc}
 x_{i}(x_{1}x_{2}\cdots x_{n}) = (x_{1}x_{2}\cdots x_{n})x_{i} & (i=1,\ldots, n)\\
 x_{i}x_{i+1}=x_{i+1}x_{i} & (i=1,\ldots, n-1) \\
 x_{n}x_{1} = x_{1}x_{n} &
 \end{array}
\right \rangle.
\]

We show that we are able to modify $\widehat{S}_{n}$ $(n>4)$ to obtain an abelian surface link of rank $n$, by adding 1-handle abelian linkings $H(n)= \frac{1}{2}(n^{2}-5n+2)$ times. 

As the first step, let us consider the case $n=5$. Let $S_{5}$ be the surface link obtained from $\widehat{S}_{5}$ by adding an abelian linking between $x_{5}$ and $x_{2}$. From the presentation of $G_{5}$, we can directly check that the link group of $S_{5}$ is a free abelian group of rank five.

In general case, first we add 1-handle abelian linkings between $x_{n}$ and $x_{i}$ for $i=2,\ldots,n-3$. From the relation
$x_{n}(x_{1}x_{2}\cdots x_{n}) =(x_{1}x_{2}\cdots x_{n})x_{n}$, this makes $x_{n}$ central. 
Then we add a 1-handle abelian linking between $x_{1}$ and $x_{n-1}$, and let $\bar{S}_{n}$ be the resulting surface link. Then the link group $\bar{G}_{n}$ of $\bar{S}_{n}$ decomposes as
$\bar{G}_{n}= \langle x_{n} \rangle \times G_{n-1} = \Z \times G_{n-1}$.
By induction, we are able to construct an abelian surface link $S_{n}$ from  $\bar{S}_{n}$ by adding 1-handle abelian linkings $H(n-1)$ times, so we get $H(n)=(n-4)+ 1 + H(n-1) = \frac{1}{2}(n^{2}-5n+2)$. 
The genus of $S_{n}$ is $(n+1)+H(n)=\frac{1}{2}(n^{2}-3n+4)$.
\end{proof}

\section*{acknowledgements}
The first author was partially supported by JSPS Postdoctoral Fellowships for Research Abroad. 
The second author was supported by JSPS Research Fellowships for Young Scientists ($24 \cdot 9014$), and iBMath through the fund for Platform for Dynamic Approaches to Living System from MEXT. 
We gratefully thank J.\ A.\ Hillman for pointing out in his review the error in Theorem \ref{theorem:bound} in the published version.

\end{document}